\documentclass[a4paper]{article}
\usepackage{amsthm,amssymb,amsmath,enumerate} 
\usepackage{graphicx}

\newcommand{\N}{\ensuremath{\mathbb{N}}}

\newcommand{\bbN}{\ensuremath{\mathbb{N}}}

\newcommand{\bbZ}{\ensuremath{\mathbb{Z}}}
\newcommand{\Pcal}{\ensuremath{\mathcal{P}}}
\newcommand{\Rcal}{\ensuremath{\mathcal{R}}}

\newcommand{\wtS}{\widetilde{S}}
\newcommand{\wtG}{\widetilde{G}}

\newcommand{\wts}{\widetilde{s}}

\newcommand{\free}{\ast\!\! }

\DeclareMathOperator{\defi}{:\!=}

\theoremstyle{definition}

\theoremstyle{plain}
\newtheorem{theorem}{Theorem}[subsection]
\newtheorem{lemma}[theorem]{Lemma}

\newtheorem{thm}[theorem]{Theorem}

\newtheorem{remark}[theorem]{Remark}

\newtheorem{prob}{Problem}

\newtheorem{conj}{Conjecture}

\theoremstyle{remark}

\newenvironment{txteq*}
  {
    \begin{equation*}
    \begin{minipage}[t]{0.85\textwidth} 
    \em                                
  }
  {\end{minipage}\end{equation*}\ignorespacesafterend}

\title{Hamilton circles in Cayley graphs
\thanks
{{\it Key Words}: Cayley graphs, Hamilton circles, Infinite graph, Infinite groups}
\thanks {{\it Mathematics Subject Classification 2010}: 05C25, 05C45, 05C63, 20E06, 20F05, 37F20.
 }}
\author{Babak Miraftab \and Tim R\"uhmann \smallskip \and Department of Mathematics\\ University of Hamburg}
\date{\today}

\begin{document}
\maketitle

\begin{abstract}
For  locally finite infinite graphs the notion of  Hamilton cycles can be extended to Hamilton circles, homeomorphic images of~$S^1$ in the Freudenthal compactification.
In this paper we prove of a sufficient condition for the existence of Hamilton circles in locally finite Cayley graphs.  
\end{abstract}

\section{Introduction}

In 1969, Lov\'{a}sz, see \cite{BabaiLovasz}, conjectured that every finite connected vertex-transitive graph contains a Hamilton cycle except five known counterexamples.
As the  Lov\'{a}sz conjecture is still open, one might instead try to solve the, possibly easier, Lov\'{a}sz conjecture for finite Cayley graphs which states: 
Every finite Cayley graph with at least three vertices contains a Hamilton cycle.  
Doing so enables the use of group theoretic tools and more over one can ask for what generating sets a particular group contains a Hamilton cycle.  
There are a vast number of papers regarding the study of Hamilton cycles in finite Cayley graphs, see \cite{Dragan, Durnberger1983, commutatorsubgroup, wittedigraphs, p-group}, for a survey of the field see~\cite{wittesurvey}. 

As cycles are always finite, we need a generalization of Hamilton cycles for infinite graphs. 
We follow the topological  approach  of  \cite{diestelBook10noEE, TopSurveyI, Ends}, which extends  Hamilton cycles  in  a  sensible  way  by using the circles in the Freudenthal compactification~$| \Gamma | $ of a~$\Gamma$ graph  as infinite cycles.
There are already results on Hamilton circles in general infinite locally finite graphs, see \cite{AgelosFleisch, hamlehpott, Karl1, Karl2}.

It is worth remarking that the weaker version of the Lovasz's conjecture does not hold true for infinite groups.
For example it is straight forward to check that the Cayley graph of any free group with the standard generating set does not contain Hamilton circles, as they are trees. 

It is a known fact that every locally finite graph needs to be 1-tough to contain a Hamilton circle, see~\cite{AgelosFleisch}.  
Thus a way to obtain infinitely many Cayley graphs with no Hamilton circle is to amalgamate more than~$k$ groups over a subgroup of order~$k$. 
In 2009, Georgakopoulos~\cite{AgelosFleisch} asked if avoiding this might be enough to force the existence of Hamilton circles in locally finite graphs and proposed the following problem:

\begin{prob}{\rm\cite[Problem 2]{AgelosFleisch}} 
\label{conj1}
Let~$\Gamma$ be a connected Cayley graph of a finitely generated group.
Then~$\Gamma$ has a Hamilton circle unless there is a~$k\in\mathbb N$
such that the Cayley graph of~$\Gamma$ is the amalgamated product of more than~$k$
groups over a subgroup of order~$k$.
\end{prob}
\noindent In Section~\ref{disprove agelos} we show that Problem~\ref{conj1} is false. 

For a one-ended graph~$\Gamma$ it suffices to find a spanning two-way infinite path, a \emph{double ray},  to find a Hamilton circle of~${|\Gamma|}$.   
In 1959 Nash-Williams \cite{nash} showed that any Cayley graph of any infinite finitely generated abelian group admits a spanning double ray.
So we can say that Nash-Williams~\cite{nash} was one of the first who proved the existence of  Hamilton circles in an infinite class of Cayley graphs even though at the time this notion of Hamilton circles was not yet properly defined. 
We extend this result by showing that any Cayley graph of any finitely generated abelian group, besides~$\bbZ$ generated by~$\{{\pm}1\}$, contains a Hamilton circle in Section~\ref{abgrp}.
We extend this result also to an even larger class of infinite groups, see Section~\ref{results} for the details.

\section{Preliminaries}
\label{prelim}

\subsection{Groups}
\label{prelimgroup}
Throughout this paper~$G$ will be reserved for groups. 
For a group~$G$ with respect to the generating set~$S$, i.e.~$G= \langle S \rangle$, we denote the Cayley graph of~$G$ with respect to~$S$ by~$\Gamma(G,S)$ unless explicitly stated otherwise. 
For a set~$T \subseteq G$ we set~$T^{\pm} \defi T\cup T^{-1}$.
Through out this paper we assume that all generating sets are symmetric, i.e. whenever~$s \in S$ then~$s^{-1} \in S$.  
Thus if we add an element~$s$ to a generating set~$S$, we always also add the inverse of~$s$ to~$S$ as well. 

Suppose that $G$ is an abelian group.
A finite set  of elements  $\{g_i\}_{i=1}^n$ of $G$ is called \emph{linear dependent} if there exist integers $\lambda_i$ for $i=1,\ldots,n$, not all zero, such that $\sum_{i=1}^{n}\lambda_ig_i=0$.
A system of elements that does not have this property is called \emph{linear independent}. 
It is an easy observation that a set  containing elements of finite order is linear dependent.
The \emph{rank} of an abelian group is the size of a maximal independent set.
This is exactly the rank the torsion free part, i.e if~$G = \bbZ^n \oplus G_0$ then rank of~$G$  is $n$, where~$G_0$ is the torsion part of~$G$. 

Let~$G_1$ and~$G_2$ be two groups with subgroups~$H_1$ and~$H_2$ respectively such that there is an isomorphism~${\phi\colon H_1\to H_2}$.
The \emph{free product with amalgamation} is defined as

$$G_1\free_{H_1} G_2 \defi  \langle S_1\cup S_2\mid R_1\cup R_2\cup H_1\phi^{-1}(H_1)\rangle.$$
A way to present elements of a free product with amalgamation is the Britton's Lemma:

\begin{lemma} {\rm \cite[Theorem 11.3]{bogo}}
Let $G_1$ and $G_2$  be two groups with subgroups~${H_1\cong H_2}$ respectively. 
Let~$T_{i}$ be a left transversal\footnote{A \emph{transversal} is a system of representatives of left cosets of~$H_i$ in~$G_i$ and we always assume that~$1$ belongs to it.} of~$H_i$ for~${i=1,2}$.
Any element $x\in G_1\free_{H} G_2$ can be uniquely written in the form  $x=x_0x_1\cdots x_n$ with the following:
\begin{itemize}
\item[{\rm(i)}] $x_0\in H_1$.
\item[{\rm(ii)}]$x_j\in T_1\setminus 1$ or $x_i\in T_2 \setminus 1$ for $j\geq 1$ and the consecutive terms $x_j$ and $x_{j+1}$ lie in distinct transversals.
\end{itemize}
\end{lemma}

Let $G=\langle S\mid R\rangle$ be a group with subgroups~$H_1$ and~$H_2$ in such a way that there is an isomorphism~${\phi\colon H_1\to H_2}$.
We now insert a new symbol~$t$ not in~$G$ and we define the \emph{HNN-extension} of~${G\ast_{H_1}}$ as follows:
$${G\ast_{H_1}} \defi \langle S, t\mid R\cup t^{-1}H_1t \phi(H_1)^{-1}\rangle.$$

Throughout this paper we assume that any generating set~${S=\{s_1,\ldots,s_n\}}$ is \emph{minimal} in the following sense: 
Each~$s_i \in S$ cannot be generated by~${S \setminus \{s_i\}}$, i.e. we have that~${s_i \notin \langle s_j\rangle_{j \in \{1,\ldots, n\}\setminus \{i\}}}$.
We may do so because say ~$S^\prime \subseteq S$ is a minimal generating set of~$G$.  
If we can find a Hamilton circle~$C$ in~$\Gamma (G,S^\prime)$, then this circle~$C$  will still be a Hamilton circle in~$\Gamma (G,S)$. 
For this it is important to note that the number of ends of~$G$ and thus of~$\Gamma(G,S^\prime)$ does not change with changing the generating set to~$S$ by {\rm{\cite[Theorem 11.23]{mei}}}, as long as~$S$ is finite, which will always be true in this paper. 

We now cite a structure for finitely generated groups with two ends. 

\begin{thm}{\rm\cite[Theorem 5.12]{ScottWall}}\label{stallings} Let~$G$ be a finitely generated group. Then the following statements are 
  equivalent.
  \begin{enumerate}[\rm (i)]
    \item The number of ends of~$G$ is~$2$.
    \item $G$ has an infinite cyclic subgroup of finite index.
    \item $G=A \free_C B$ and~$C$ is finite and~$[A:C]=[B:C]=2$ or~$G=C\free_C$ with~$C$ is finite.
  \end{enumerate}
\end{thm}

\subsection{Graphs}
\label{prelimgraph}

Throughout this paper~$\Gamma$ will be reserved for graphs. 
In addition to the notation of paths and cycles as sequences of vertices such that there are edges between successive vertices we use the notation of \cite{commutatorsubgroup, wittesurvey} for constructing Hamilton paths and Hamilton cycles and circles which uses edges rather than vertices.
We give some basic examples of this definition here: 
For that let~$g$ and~$s_i,~i \in \bbZ$, be elements of some group. 
In this notation $g[s_1]^k$ denotes the concatenation of~$k$ copies of~$s_1$ from the right starting from~$g$ which translates to the path~$g,(gs_1),\ldots,(gs_1^k)$ in the usual notation. 
Analogously~${[s_1]^k g}$ denotes the concatenation of~$k$ copies of~$s_1$ starting again from~$g$ §from the left.
In addition~$g[s_1,s_2,\ldots]$ translates to be the ray~$g,(gs_1),(gs_1s_2),\ldots$ and 
$$[ \ldots , s_{-2},s_{-1}]g[s_1,s_2, \ldots ]$$
translates to be the double ray 
$$\ldots, (gs_{-2}s_{-1}),(gs_{-1}),g,(gs),(gs_1s_2),\ldots$$ 
When discussing rays we extend the notation of~$g[s_1,\ldots, s_n]^k$ to~$k$ being countably infinite and write~$g[s_1,\ldots, s_2]^\bbN$ and the analogue for double rays.  
Sometimes we will use this notation also for cycles. 
Stating that~$g[c_1,\ldots,c_k]$ is a cycle means that~$g[c_1,\ldots, c_{k-1}]$ is a path and that the edge~$c_k$ joins the vertices~${g c_1\cdots c_{k-1}}$ and~${g}$.

For a graph~$\Gamma$ let the induced subgraph on the vertex set~$X$ be called~$\Gamma[X]$.
Throughout this paper we use Theorem \ref{stallings} to characterize the structure of two ended groups, see Section \ref{HCsection}  for more details. 
It is still important to pay close attention to the generating sets for those groups though, as the following example shows. 
Take two copies of $\bbZ_2$, with generating sets~$\{a\}$ and  $\{b\}$, respectively.
Now consider the free product of them. 
It is obvious that this Cayley graph with generating set $\{a,b\}$ does not contain a Hamilton circle. 
The free product of $\bbZ_2$ with itself, on the other hand, is isomorphic to $D_\infty$, which can represented  by~${D_\infty = \langle a,b \mid a^2 =1, aba=b^{-1} \rangle }$. 
It is easy to see that the Cayley graph of $D_\infty$ with this generating set contains a Hamilton circle. 

For a graph~$\Gamma$ we denote the Freudenthal compactification of~$\Gamma$ by~$|\Gamma|$. 
A homeomorphic image of~$[0,1]$ in the topological space~$|\Gamma |$ is called \emph{arc}. 
A \emph{Hamilton arc} in~$\Gamma$ is an arc including all vertices of~$\Gamma$. 
So a Hamilton arc in a graph always contains all ends of the graph. 
By a \emph{Hamilton circle} in~$\Gamma$, we mean  a homeomorphic image of the unit circle in~$|\Gamma|$ containing all vertices of~$\Gamma$.
A Hamilton arc whose image in a graph is connected is a \emph{Hamilton double ray}.
It is worth mentioning that an uncountable graph cannot contain a Hamilton circle. 
To illustrate, let~$C$ be a Hamilton circle of graph~$\Gamma$. 
Since~$C$ is homeomorphic to~$S^1$, we can assign to every edge of~$C$ a rational number. 
Thus we can conclude that~$V(C)$ is countable and so~$\Gamma$ is countable. 
Hence in this paper, we assume that all groups are countable.
In addition we will only consider groups with locally finite Cayley graphs in this paper so we assume that all generating sets~$S$ will be finite.\footnote{For not locally finite graphs the Freudenthal compactification, also called {\sc Vtop}, of a graph~$\Gamma$ is less useful as the space~$|\Gamma|$ might not be Hausdorf, which would be nice property to have. One might consider different topologies like {\sc Etop} or {\sc Mtop} for instead of {\sc Vtop} for non locally finite~$\Gamma$.}

\section{Hamilton circle}
\label{HCsection}
In this section we prove sufficient conditions for the existence of Hamilton circles in Cayley graphs. 
In Section~\ref{abgrp} we take a look at abelian groups. 
Section~\ref{tools} contains basic lemmas and structure theorems used to prove our main results which we prove in the Section~\ref{results}. 

\subsection{Abelian Groups}
\label{abgrp}
In the following we will examine abelian groups as a simple starting point for studying Hamilton circles in infinite Cayley graphs. 
Our main goal in this section is to extend a well-known theorem of Nash-Williams from one-ended abelian groups to two ended abelian groups by a simple combinatorial argument. 
First,  we cite a known result for finite abelian groups.

\begin{lemma}\label{HCabelian}{\rm{\cite[Corollary 3.2]{Marusic}}} 
Let $G$ be a finite abelian group with at least three elements. 
Then any Cayley graph of $G$ has a Hamilton cycle.
\end{lemma}

\noindent In the following we extend the previously mentioned theorem of Nash-Williams from finitely one-ended abelian groups to the maximal set of finitely generated abelian groups. 
For that we first state the theorem of Nash-Williams.

\begin{thm}\label{Nash}{\rm{\cite[Theorem 1]{nash}}} 
Let~$G$ be a finitely generated abelian group with exactly one end. 
Then any Cayley graph of~$G$ has a Hamilton circle.
\end{thm}

It is obvious that the maximal class of groups to extend Theorem \ref{Nash} to cannot contain~$\Gamma(\mathbb{Z},\{\pm1\})$, as this it cannot contain a Hamilton circle. 
In Theorem \ref{Z} we prove that this is the only exception.    
      
\begin{thm}\label{Z}
Let~$G$ be an infinite finitely generated abelian group. 
Then any Cayley graph of~$G$ has a Hamilton circle except~$\Gamma(\mathbb{Z},\{1\})$.
\end{thm}
      
\begin{proof}
By the fundamental theorem of finitely generated abelian groups \cite[5.4.2]{scott}, one can see that~$G\cong \mathbb Z^n\oplus G_0$ where~$G_0$ is the torsion part of~$G$ and~$n\in\mathbb N$.
It follows from~\cite[lemma 5.6]{ScottWall} that the number of ends of~$\bbZ^n$ and~$G$ are equal. 
We know that the number of ends of~$\bbZ^n$ is one if~$n \geq 2$ and two if~$n =1$. 
By Theorem \ref{Nash} we are done if~${n  \geq 2}$.  
So we can assume that~$G$ has exactly two ends.

Now suppose that~$S=\{s_1,\ldots,s_\ell\}$ generates~$G$. 
Without loss generality assume that the order of~$s_1$ is infinite.
Let~$i$ be the smallest natural number such that~$s_2^{i+1}\in\langle s_1\rangle$.
Since the rank of~$G$ is one, we can conclude that~$\{s_1,s_2\}$ are dependent and thus such an~$i$ exists.
In the following we define a sequence of double rays. 
We start with the double ray~${R_1 =  [s_1^{-1}]^\bbN 1 [s_1]^\bbN}$.  
Now we replace every other edge of~$R_1$ by a path to obtain a double ray spanning~${\langle s_1, s_2 \rangle}$.
The edge~$1s_1$ will be replaced by the path~${[s_2]^i[s_1][s_2^{-1}]^{i}}$.
We obtain the following double ray:
 \[
 R_2=\cdots  [s_2]^{-i}[s_1^{-1}][s_2]^{i}[s_1^{-1}]1[s_2]^{i}[s_1][s_2^{-1}]^{i}[s_1] \cdots 
 \]  

Note that~$R_2$ spans~$\langle s_1,s_2 \rangle$. 
We will now repeat this kind of construction for additional generators. 
For simplicity we denote~$R_2$ by~${[\ldots ,y_{-2},y_{-1}]1[y_1,y_2,\ldots]}$ with~${y_k\in\{s_1,s_2\}^{\pm}}$ for every~$k \in \mathbb Z \setminus \{0\}$.  
As above let~$j \in \N$ be minimal such that~${s_3^{j+1} \in\langle s_1,s_2 \rangle}$. 
We now define the double ray
 \[R_3= \cdots [s_3^{-1}]^j[y_{-2}][s_3^{j}][y_{-1}]1[s_3^j][y_1][s_3^{-1}]^j[y_2] \cdots .\]  
We now repeat the process until we have defined the double ray~$R_{\ell-1}$, say
\[ {{R_{\ell-1}}=[\ldots,x_{-2},x_{-1}]1[x_1,x_2,\ldots]}\]
with~${x_k\in\{s_1,\ldots,s_{\ell-1}\}^{\pm}}$ for every~${k\in\mathbb Z\setminus\{0\}}$.
Now let~$u$ be the smallest natural number such that~$s_\ell^{u+1}\in\langle s_1,\ldots,s_{\ell-1}\rangle$. 
Now, put 
\[ \mathcal{P}_1=\cdots[s_\ell^{-1}]^{u-1}[x_{-2}][s_\ell]^{u-1}[x_{-1}]1[s_\ell]^{u-1}[x_1][s_\ell^{-1}]^{u-1}[x_2]\cdots \]
and 
 \[\mathcal{P}_2=[\ldots ,x_{-2},x_{-1}]s_\ell^{u}[x_1,x_2,\ldots].\]    
It is not hard to see that~$\mathcal P_1\cup\mathcal P_2$ is a Hamilton circle of~$\Gamma(G,S)$.
\end{proof}

\begin{remark}
One can prove Theorem {\rm \ref{Nash}} by same the arguments used in the above proof of Theorem {\rm \ref{Z}}. 
\end{remark}

\subsection{Structure Tools}
\label{tools}

In this section we assemble all the most basic tools to prove our main results. 
The our most important tools are Lemma~\ref{Cylinder} and Lemma~\ref{altern_cylinder}. 
In both Lemmas we  prove that a given graph~$\Gamma$ contains a Hamilton circle if it admits a partition of  its vertex set into infinitely many finite sets~${X_i, i \in \bbZ}$,  all of the same size which contain some special cycle and such that~$\Gamma$ connects these cycles in a useful way, see Lemma \ref{Cylinder} and \ref{altern_cylinder} for details.

\begin{lemma}\label{Cylinder}
Let~$\Gamma$ be a graph that admits a partition of its vertex set into finite sets~$X_i, ~i \in \bbZ$, fulfilling the following conditions:

\begin{enumerate}[\rm (i)]
\item~$\Gamma[X_i]$ contains a Hamilton cycle~$C_i$ or~$\Gamma[X_i]$ is isomorphic to~$K_2$. 
\item For each~${i \in \bbZ}$ there is a perfect matching between~$X_i$ and~$X_{i+1}$. 
\item There is a~$k \in \bbN$ such that for all~$i,j \in \bbZ$ with~${|i -j| \geq k}$ there is no edge in~$\Gamma$ between~$X_i$ and~$X_{j}$.
\end{enumerate}
Then~$\Gamma$ has a Hamilton circle.
\end{lemma}

\begin{proof}
By (i) we know that each~$X_i$ is connected and so we conclude from the structure given by (ii) and (iii) that~$\Gamma$ has exactly two ends. 
In addition note that~$|X_i|=|X_j|$ for all~${i,j \in \bbZ}$. 
First we assume that~${\Gamma[X_i}]$ is just a~$K_2$. 
It follows directly that~$\Gamma$ is spanned by the double ladder, which is well-known to contain a Hamilton circle.
As this double ladder shares its ends with~$\Gamma$, this Hamilton circle is also a Hamilton circle of~$\Gamma$. 

Now we assume that~$|X_i| \geq 3$. 
Fix an orientation of each~$C_i$. 
The goal is to find two disjoint spanning doubles rays in~$\Gamma$. 
We first define two disjoint rays belonging to same end, say for all the~$X_i$ with~${i \geq 1}$.   
Pick two vertices~$u_1$ and~$w_1$ in~$X_1$. 
For~$R_1$ we start with~$u_1$ and move along~$C_1$ in the fixed orientation of~$C_1$ till the next vertex on~$C_1$ would be~$w_1$, we then instead of moving along we move by  the given matching edge to~$X_2$.
We take this  to be a the initial part of~$R_1$.
We do the analog for~$R_2$ by starting with~$w_1$ and moving also along~$C_1$ in the fixed orientation till the next vertex would be~$u_1$, then move to~$X_2$.   
We repeat the process of starting with some~$X_i$ in two vertices~$u_i$ and~$w_i$, where~$u_i$ is the first vertex of~$R_1$ on~$X_i$ and~$w_i$ the analog for~$R_2$.  
We follow along the fixed orientation on~$C_i$ till the next vertex would be~$u_i$ or~$w_i$, respectively.
Then we move to~$X_{i+1}$ by the giving matching edges. 
One can easily see that each vertex of~$X_i$ for~$i \geq 1$ is contained exactly either in~$R_1$ or~$R_2$.
By moving from~$u_1$ and~$w_1$ to~$X_0$ by the matching edges and then using the same process but moving from~$X_i$ to~$X_{i-1}$ extents the rays~$R_1$ and~$R_2$ into two double rays. 
Obviously those double rays are spanning and disjoint. 
As~$\Gamma$ has exactly two ends it remains to show that~$R_1$ and~$R_2$ have a tail in each end. 
By (ii) there is a~$k$ such that there is no edge between any~$X_i$ and~$X_j$ with~$|i-j|\geq k$.
The union~$\bigcup_{i=\ell}^{\ell+k} X_i,~\ell \in \bbZ$, separates~$\Gamma$ into two components such that~$R_i$ has a tail in each component, which is sufficient.  
\end{proof}

\noindent Next we prove a slightly different version of Lemma \ref{Cylinder}. 
In  this version we split each~$X_i$ into any upper and lower part,~$X_i^+$ and~$X_i^-$, and assume that we only find  a perfect matching between upper and lower parts of adjacent partition classes, see Lemma \ref{altern_cylinder} for details. 

\begin{lemma}\label{altern_cylinder}
Let~$\Gamma$ be a graph that admits a partition of its vertex set into finite sets~$X_i, i \in \bbZ$ with~${|X_i|\geq 4}$ fulfilling the following conditions:

\begin{enumerate}[\rm (i)] 
\item~$X_i = X_i^+ \cup X_i^-$, such that~$X_i^+ \cap X_i^- = \emptyset$ and~$|X_i^+| = |X_i^-|$
\item~$\Gamma[X_i]$ contains an Hamilton cycle~$C_i$ which is alternating between~$X_i^-$ and~$X_i^+$.\footnote{Exactly every other element of~$C_i$ is contained in~$X_i^-$.}
\item For each~${i \in \bbZ}$ there is a perfect matching between~$X^+_i$ and~$X^-_{i+1}$. 
\item There is a~$k \in \bbN$ such that for all~$i,j \in \bbZ$ with~${|i -j| \geq k}$ there is no edge in~$\Gamma$ between~$X_i$ and~$X_{j}$.
\end{enumerate}
Then~$\Gamma$ has a Hamilton circle.
\end{lemma}

The proof of Lemma \ref{altern_cylinder} is very closely related to the proof of Lemma \ref{Cylinder}. 
We still give the complete proof for completeness.

\begin{proof}
By (i) we know that each~$X_i$ is connected and so we conclude from the structure given by (ii) and (iii) that~$\Gamma$ has exactly two ends. 
In addition note that~$|X_i|=|X_j|$ for all~${i,j \in \bbZ}$. 

Fix an orientation of each~$C_i$. 
The goal is to find two disjoint spanning doubles rays in~$\Gamma$. 
We first define two disjoint rays belonging to the same end, say for all the~$X_i$ with~${i \geq 0}$.   
Pick two vertices~$u_1$ and~$w_1$ in~$X_1^-$. 
For~$R_1$ we start with~$u_1$ and move along~$C_1$ in the fixed orientation of~$C_1$ till the next vertex on~$C_1$, then instead of moving along~$C_1$ we move to~$X_2^-$ by  the given matching edge. 
Note that as~$w_1$ is in~$X_1^-$ and because each~$C_i$ is alternating between~$X_i^-$ and~$X_i^+$  this is possible.  
We take this  to be a the initial part of~$R_1$.
We do the analog for~$R_2$ by starting with~$w_1$ and moving also along~$C_1$ in the fixed orientation till the next vertex would be~$u_1$, then move to~$X_2$.   
We repeat the process of starting with some~$X_i$ in two vertices~$u_i$ and~$w_i$, where~$u_i$ is the first vertex of~$R_1$ on~$X_i$ and~$w_i$ the analog for~$R_2$.  
We follow along the fixed orientation on~$C_i$ till the next vertex would be~$u_i$ or~$w_i$, respectively.
Then we move to~$X_{i+1}$ by the giving matching edges. 
One can easily see that each vertex of~$X_i$ for~$i \geq 1$ is contained exactly either in~$R_1$ or~$R_2$.
By moving from~$u_1$ and~$w_1$ to~$X_0^+$ by the matching edges and then using the same process but moving from~$X_i^-$ to~$X_{i-1}^+$ extents the rays~$R_1$ and~$R_2$ into two double rays. 
Obviously those double rays are spanning and disjoint. 
As~$\Gamma$ has exactly two ends it remains to show that~$R_1$ and~$R_2$ have a tail in each end. 
By (ii) there is a~$k$ such that there is no edge between any~$X_i$ and~$X_j$ with~$|i-j|\geq k$ the union~$\bigcup_{i=l}^{\ell+k} X_i, ~l \in \bbZ$ separates a~$\Gamma$ into two components such that~$R_i$ has a tail in each component, which is sufficient.  
\end{proof} 
 
\begin{remark}
\label{arcaltern_cylinder}
It is easy to see that one can find a Hamilton double ray instead of a Hamilton circle in Lemma $\ref{Cylinder}$ and Lemma $\ref{altern_cylinder}$.
Instead of starting with two vertices and following in the given orientation to define the two double rays, one just starts in a single vertex and follows the same orientation. 
\end{remark}

The following lemma is  one of our main tools in proving the existence of Hamilton circles in Cayley graphs. 
It is important to note that the restriction, that~${S \cap H = \emptyset}$, which looks very harsh at first glance, will not be as restrictive in the later parts of this paper.
In most cases we can turn the case~${S \cap H \neq \emptyset}$ into the case~${S \cap H = \emptyset}$ by taking an appropriate quotient.

\begin{lemma}
\label{ZigZag}
Let~$G=\langle S\rangle$ and~$\widetilde G=\langle \widetilde S\rangle$ be  finite groups with non-trivial subgroups~$H\cong \widetilde H$ of indices two such that~$S\cap H=\emptyset$ and such that~${\Gamma(G,S)}$ contains a Hamilton cycle. 
Then the following statements are true. 
\begin{enumerate}[\rm (i)]
\item $\Gamma({G \free_H \wtG},S\cup \wtS)$ has a Hamilton circle.
\item $\Gamma({G \free_H \wtG},S\cup \wtS)$ has a Hamilton double ray.
\end{enumerate}
\end{lemma}

To prove Lemma \ref{ZigZag} we start by finding some general structure given by our assumptions.
This structure will make it possible to use Lemma \ref{altern_cylinder} and Remark~\ref{arcaltern_cylinder} to prove the statements (i) and (ii). 

\begin{proof}
First we define~$\Gamma \defi \Gamma({G \free_H \wtG},S\cup \wtS)$.
Let~${s \in S \setminus H}$ and let~$\wts$ be in~${\wtS\setminus \widetilde H}$.
By our assumptions~${\Gamma(G,S)}$ contains a Hamilton cycle, say~$C_0=1[c_1,\ldots, c_k]$. 
It follows from~${S \cap H = \emptyset}$ that~$C_0$ is alternating between~$H$ and the right coset~$Hs$. 
For each~${i \in \bbZ}$  we now define the graph~$\Gamma_i$. 
\begin{align*}
 \textnormal{For } i \geq 0 \textnormal{ we define } \Gamma_i &\defi \Gamma[{H (s\wts)^i \cup H (s\wts)^i s}] \\
\textnormal{ and for } i \leq -1 \textnormal{ we define } \Gamma_i &\defi  \Gamma[{H \wts (s \wts)^{-i-1} \cup H (\wts s)^{-i}}].
\end{align*} 
By our assumptions we know that~$C_0$ is a Hamilton cycle of~$\Gamma_0$. 
We now define Hamilton cycles of~$\Gamma_i$ for all~$i \neq 0$. 

\begin{align*}
\textnormal{For } i \geq 1 \textnormal{we define } C_i &\defi (s\wts)^i [c_1,\ldots, c_k] \\
\textnormal{ and for } i \leq -1 \textnormal{ we define } C_i &\defi (\wts s)^{-i} [c_1,\ldots, c_k].
\end{align*}
To show that~$C_i$ is a Hamilton cycle of~$\Gamma_i$ it is enough to show that~$C_i$ is a cycle and that~$C_i$ contains no vertex outside of~$\Gamma_i$, because all cosets of~$H$ have the same size and because~$C_0$ is a Hamilton cycle of~${\Gamma_0=\Gamma(G,S)}$. 

For~$i \geq 1$ we first show that~$C_i$ is a cycle. 
It follows directly from the fact that~$C_0$ is a cycle that in~$\Gamma$ each~$C_i$ is closed.\footnote{$\Gamma$ contains the edge between the image of~$c_1$ and~$c_k$ for each~$C_i$.} 
Assume for a contraction that~${(s \wts)^i c_0 \cdots c_j = (s \wts)^i c_0 \cdots c_\ell}$ for some~${j < \ell}$. 
This contracts that~$C_0$ is a cycle as it is equivalent to~${1 = c_{j+1} \cdots c_\ell}$.

It remains to show that every vertex of~$C_i$ is contained in~$\Gamma_i$.
Since~$H$ is a normal subgroup of both~$G$ and~$\widetilde{G}$, the elements~$s$ and~$\wts$ commute with~$H$. 
As each vertex~$v \defi c_0\ldots c_j$ is contained in either~$H$ or~$Hs$ we can conclude that~${(s \wts)^i v \in (s \wts)^i H= H (s \wts)^i}$ or~${ (s \wts)^i v\in (s \wts)^i Hs = H (s \wts)^i s}$.

In the following we give some easy observations on the structure of the~$C_i$'s. 
First note that~${C_i \cap C_j = \emptyset}$ for~$i \neq j$ and also that the union of all~$C_i$'s contains all the vertices of~$\Gamma$.
In addition note that each~$C_i$ is alternating between two copies of~$H$ as~$C_0$ was alternating between cosets of~$\Gamma_0$. 
Finally note that by the structure of~$\Gamma$ there is no edge between any~$\Gamma_i$ and~$\Gamma_j$ with~$|i-j| \geq 2$ in~$\Gamma$.

By the structure of~$\Gamma$ for~$i \geq 0$ we get a perfect matching between~$C_i \cap H (s \wts)^i s$ and~$C_{i+1} \cap H (s \wts)^{i+1}$ by~$\wts$.

By an analog argument one can show that for~${i < 0}$ we get a similar structure and the desired perfect matchings. 

The statement (i) now follows by Lemma \ref{altern_cylinder}.
Analog statement (ii) follows by Remark \ref{arcaltern_cylinder}.
\end{proof}

We now recall two known statements about Hamilton cycles on finite groups,  which we then will first combine and finally generalize to infinite groups. 
For that let us first recall some definitions. 
A group~$G$ is called \emph{Dedekind}, if every subgroup of~$G$ is normal in~$G$. 
If a Dedekind groups~$G$ is also non-abelian, it is called a \emph{Hamilton group}.

\begin{lemma}\label{hamilton}{\rm \cite{ChenQuimpo}}
Any Cayley graph of  a Hamilton group $G$ has a Hamilton cycle.
\end{lemma}

In addition we know that all finite abelian groups also contain Hamilton cycles by Lemma \ref{HCabelian}. 
In the following remark we combine these two facts. 

\begin{remark}\label{dedekind}
Any Cayley graph of a finite Dedekind group of order at least three contains a Hamilton cycle.
\end{remark}

\subsection{Main Results}
\label{results}
In this section we prove our main results. 
For that let us recall that by Theorem~\ref{stallings} we know that there every two ended group either a free product with amalgamation over a finite subgroup of index two or an HNN-extension over a finite subgroup. 
Now we prove our first main result, Thereom~\ref{infinite semihamilton}, which deals with the first type of groups. 
To be more precise we use Remark~\ref{dedekind} to prove that the free product of a Dedekind group with a second group with amalgamation over the subgroup of index two in both of those groups contains a Hamilton circle. 

\begin{thm}
\label{infinite semihamilton}
Let~$G=\langle S\rangle$ and~$\widetilde G=\langle \widetilde S\rangle$ be  two finite groups with  non-trivial subgroups~$H\cong \widetilde {H}$ of  indices two and such that~$G$ is a Dedekind group.   
Then~$\Gamma({G \free_{H} \wtG},S\cup \wtS)$ has a Hamilton circle.
\end{thm}

\begin{proof}
First, it follows from Remark \ref{dedekind} that~${\Gamma(G,S)}$ has a Hamilton cycle.
If all generators of~$S=\{s_1,\ldots,s_n\}$ lie outside~$H$, then Lemma \ref{ZigZag} completes the proof.
So let~${s_n \in S \setminus H}$ and let~${\wts \in \wtS\setminus \widetilde H}$.
Suppose that~$S^\prime \defi {\{s_1,\ldots, s_i \}}$ is a maximal set of generators of~$S$ contained in~$H$ and set~$L:=\langle S^\prime \rangle$. 
First note that~$L$ is a normal subgroup of~$G$.
We now have two cases, either~${H = L}$ or~${L \neq H}$. 
We may assume that~${H \neq L}$ as otherwise we can find a Hamilton circle of~$\Gamma({G \free_{H} \wtG},S\cup \wtS)$ by Lemma~\ref{Cylinder} as~$H$ is a Dedekind group and thus~$\Gamma(H,S^\prime)$ contains a Hamilton cycle. 
Because~${L \subsetneq H}$ and~$H\cong \widetilde {H}$ we conclude that there is a subgroup of~$\widetilde{H}$ that is corresponding to~$L$, call this~$\widetilde{L}$.
Let~$\Lambda$ be the Cayley graph of the group~${G/L \free_{H/L} \wtG/\widetilde{L}}$ with the generating set~$\overline{S}\cup\overline{\tilde{S}}$, where~$\overline{S}$ and~$\overline{\tilde{S}}$ the corresponding generating sets of~$G/L$ and~$\widetilde{G}/\widetilde{L}$, respectively. 
Note that every generator of the quotient group~$G/L$ lies outside of~$H/L$. 
Hence it follows from Lemma \ref{ZigZag}, that we can find a Hamilton double ray in~$\Lambda$, say~$\mathcal R$. 
Now we are going to use~$\mathcal{R}$ and construct a Hamilton circle for~${\Gamma \defi \Gamma({G \free_{H} \wtG},S\cup \wtS)}$. 
Since~$L$ is a subgroup of~$H$, we can find a Hamilton cycle in the induced subgroup of~$L$, i.e.~$\Gamma(L,S^\prime)$.
We denote this Hamilton cycle in~$\Gamma(L,S^\prime)$ by~$C=[x_1,\ldots,x_n]$. 
We claim that the induced subgraph of any coset of~$L$ of~$G \free_{H} \wtG$ contains a Hamilton cycle. 
Let~$Lx$ be an arbitrary coset of~$G\free_{H} \wtG$.  
If we start with~$x$ and move along the edges given by~$C$, then we obtain a cycle. 
We will show that this cycle lies in~$Lx$. 
Since~$L$ is a normal subgroup of both~$G$ and~$\wtG$ it implies that~$L$ is a normal subgroup of~$G \free_{H} \wtG$. 
Since~$L$ is normal, the element~$x$ commutates with the elements of~$L$ and so~$x[C]$ lies in~$Lx$ and the claim is proved. 
It is important to notice that~$\mathcal {R}$ gives a prefect mating between each two successive cosets. 
Thus we are ready to invoke the Lemma \ref{Cylinder} and this completes the proof.
\end{proof}
 
The following Theorem \ref{semidirect} proves that the second type of two ended groups also contains a Hamilton circle, given some conditions.  
 
\begin{remark}
\label{HNN}
Let us have a closer look at an HNN extension of a finite group~$C$. 
Let~${C=\langle S\mid R\rangle}$ be a finite group.
It is important to notice that every automorphism~${\phi\colon C\to C}$ gives us an HNN extension~${G=C\free_C}$. 
In particular every such HNN extension comes from an automorphism~${\phi\colon C\to C}$. 
Therefore~$C$ is a normal subgroup of $G$ with the quotient $\bbZ$, as the presentation of HNN extension~${ G=C\free_C}$ is
\[{\langle S,t\mid R,\, t^{-1}c t=\phi(c) \,\forall c\in C\rangle }.\] 
Hence $G$ can be expressed by a semidirect product~${ C\rtimes \bbZ }$ which is induced by~$\phi$.
To summarize; every two ended group with a structure of HNN extension is a semidirect product of a finite group with the infinite cyclic group. 
\end{remark} 
 
\begin{thm}\label{semidirect}
Let~$G= (H \rtimes F , X\cup Y)$ with~$F= \mathbb Z=\langle Y\rangle$ and~$H=\langle X \rangle$ and such that~$H$ is finite and~$H$ contains a Hamilton cycle.  
Then~$G$ has a Hamilton circle.
\end{thm}

\begin{proof}
Let~$C=[c_1,\ldots ,c_t]$ be a Hamilton cycle in~${\Gamma(H, X)}$.
We now make a case study about the size of~$Y$. 
    
\noindent {\bf{Case I :}} If~$|Y|=1$, then~$F= \mathbb Z=\langle y\rangle$. 
Since~$H$ is a normal subgroup of~$G$, it follows that~$g H=H g$ for each~$g \in {G}$. 
Thus the vertices of the set~$Cg$ form a cycle for every~$g \in G$.
Let~$C_g$ be the cycle of~$H g$ for all~${g \in \mathbb Z}$, and let~$\mathcal{C}$ be the set of all those cycles. 
We show that for every pair of~$g,h \in \bbZ$ we either have~$C_h \cap C_g =\emptyset$ or~$C_h =C_g$.
Suppose that~$C_g \cap C_h \neq \emptyset$. 
This means that 
\begin{align*}
c_i  y^g &= c_j y^h \\
\Leftrightarrow  c_j^{-1} c_i &=  y^{h-g}.  
\end{align*}
The order of the left hand side is finite while the order of the right hand side is infinite. 
Thus we conclude that~${y^{h-g}=1}$ which in turn yields that~${g =h}$ thus we get~$C_g=C_h$.  
We claim that every vertex is contained in~$\mathcal{C}$. 
Suppose that~$g\in G$. 
Since~$G= H\rtimes \mathbb Z$, we deduce that~$G=H \mathbb Z$.
In other words, there is a natural number~$i$ and an~$h \in \mathbb Z$ such that~$g=c_i h$ and so~$g$ lies in the cycle~$C_h$.
These conditions now allow the application of Lemma \ref{Cylinder}, which concludes this case. 

\vspace*{0,5cm}
\noindent {\bf{Case II :}} Assume that~$|Y|\geq 2$.  
By Theorem \ref{Z} there are two disjoint double rays 
\[{\mathcal R_1=[\ldots ,x_{-2},x_{-1}]1[x_1,x_2,\ldots]} \] and \[\mathcal R_2=[\ldots, y_{-2},y_{-1}]x[y_1,y_2,\ldots] \] 
where~$x_i,y_i,x\in Y^{\pm}$ such that the vertices of~$\mathcal R_1\cup\mathcal R_2$ cover all elements~$\mathbb Z$.
Since~$H$ is a normal subgroup of~$G$, we can conclude that~$g H= H g$. 
Thus the vertices of the set~$g C$ form a cycle for every~$g\in G$.
Now consider the double rays
\[P_1=\cdots[x_{-2}][c_1,\ldots, c_{t-1}] [s_{-1}]1 [c_1,\ldots, c_{t-1}] [x_1]][c_1,\ldots, c_{t-1}]\cdots \] 
and
\[P_2=\cdots[y_{-2} ][c_1,\ldots, c_{t-1}] [y_{-1}]x [c_1,\ldots, c_{t-1}] [y_1]][c_1,\ldots, c_{t-1}]\cdots.\] 
For easier notation we define~${a \defi c_1 \cdots c_{t-1}}$. 
We claim that~$P_1 \cap P_2 = \emptyset$. 
There are 4 possible cases such intersections. 
We only consider this one case, as the others are analog.
So assume to the contrary
\[ x \cdot  a y_1  \cdots  a y_{\ell_1} \cdot  c_1 \cdots  c_{\ell^\prime_1}=  a x_1  \cdots a x_{\ell_2} \cdot  c_1 \cdots  c_{\ell_2^\prime}.\]
Since~$H$ is a normal subgroup of~$G$, for every~$g\in G$  we have~${a g =g h}$ for some~${h \in H}$.
It follows that 
\begin{align*}
 x \cdot  a y_1  \cdots  a y_{\ell_1} \cdot  c_1 \cdots  c_{\ell^\prime_1}&=  a x_1  \cdots a x_{\ell_2} \cdot  c_1 \cdots  c_{\ell_2^\prime} \\
\Leftrightarrow x \cdot  y_1  \cdots   y_{\ell_1} h \cdot  c_1 \cdots  c_{\ell^\prime_1}&=  x_1  \cdots x_{\ell_2} h^{\prime} \cdot  c_1 \cdots  c_{\ell_2^\prime}  \textnormal{ for some $h,h^\prime \in H$} \\ 
\Leftrightarrow x \cdot  y_1  \cdots   y_{\ell_1} \bar{h} &=  x_1  \cdots x_{\ell_2} \bar{h}^{\prime}  \textnormal{ for some $\bar{h}, \bar{h}^\prime \in H$} \\ 
\Leftrightarrow (x_1  \cdots x_{\ell_2})^{-1} x \cdot  y_1  \cdots   y_{\ell_1}  &=   \bar{h}^{\prime} \bar{h}^{-1} 
\end{align*}
The left side of this equation again has finite order, but the right side has infinite order. 
It follows that
\begin{align*}
(x_1 \ldots x_i)^{-1} x y_1 \cdots y_j &= 1 \\
 x y_1 \cdots y_j &= x_1 \ldots x_i
\end{align*}
But this contradicts our assumption that~$\Rcal_1$ and~$\Rcal_2$ were disjoint. 
Therefore, as~${V(\mathcal P_1\cup \mathcal P_2)=V(\Gamma(G, X \cup Y))}$, the double rays~$\Pcal_1$ and~$\Pcal_2$ form the desired Hamilton circle. 
\end{proof}

\section{Multiended groups}
\label{multiended}

In this section we give a few insights into the problem of finding Hamilton circles in groups with more than two ends,  as well as showing a counter example for Problem~\ref{conj1}.   
We call a group to be a \emph{multiended group} if is has more than two ends. 
In 1993 Diestel, Jung and M\"oller \cite{DiestelJungMoeller} proved that any transitive graph with more than two ends has infinitely many ends\footnote{In this case the number of ends is uncountably infinite.} and as all Cayley graphs are transitive it follows that the number of ends of any group is either zero, one, two or infinite. 
This yields completely new challenges for finding a Hamilton circle in groups with more than two ends.
One famous example to illustrate the problems of finding a Hamilton circles in an infinite graph with infinitely many ends is the Wild Circle~\cite[Figure 8.5.1]{diestelBook10noEE}. 
Thus studying graph with more than two ends to find Hamilton circles is more complicated than just restricting one-self to two-ended groups.

\subsection{Disprove Problem 1}
\label{disprove agelos}
We now give an example of an infinite Cayley graph that disproves Problem~\ref{conj1}. 
Define~${G_1 \defi G_2 \defi  \bbZ_3 \times \bbZ_2}$.
Let~${\Gamma \defi \Gamma(G_1 \free_{\bbZ_2} G_2)}$. 
Let~${G_1 = \langle a,b \rangle}$ and~${G_2 = \langle a,c \rangle}$ where the order of~$a$ is two and the orders of~$b$ and~$c$, respectively, are three. 
In the following we show that the assertion of Problem~\ref{conj1} holds for~$\Gamma$ and we show that~$|\Gamma|$ does not contain a Hamilton circle. 

In the following we will show that~$\Gamma$ does not contain a Hamilton circle. 
For that we use the following well-known lemma and theorem.

\begin{lemma}
{\rm \cite[Lemma 8.5.5]{diestelBook10noEE}}
\label{top_conn1}
If $\Gamma$ is a locally finite connected graph, then a standard subspace
\footnote{A standard subspace of~$|\Gamma|$ is a subspace of~$|\Gamma|$ 
that is a closure of a subgraph of~$\Gamma$.} of $|\Gamma|$ is topologically connected (equivalently: arc-connected) if and only if it contains an edge from every finite cut of $\Gamma$ of which it meets both sides.
\end{lemma}

\begin{thm}{\rm \cite[Theorem 2.5]{TopSurveyI}}
\label{top_conn2}
The following statements are equivalent for sets $D \subseteq E(\Gamma)$:
\begin{enumerate}[\rm (i)]
\item $D$ is a sum of circuits in $|\Gamma|$.
\item Every vertex and every end has even degree in $D$.
\item $D$ meets every finite cut in an even number of edges.
\end{enumerate}
\end{thm}

Applying Theorem \ref{top_conn2} it is enough to show that there is no set~${D \subseteq E(\Gamma)}$ that meets every finite cut in an even number and every vertex twice.
By using Lemma \ref{top_conn1} we can further conclude that, because circles are arc-connected, that such a~$D$ would have to meet every finite cut in at least one edge. 

We now assume for a contraction that there is a Hamilton circle in~$\Gamma$, so we assume that there is a~$D \subseteq E(\Gamma)$ that meets every finite cut in even and at least one edge, and induces degree two at every vertex.  
In the following we will now study two cases. 
In each case we will consider a few finite cuts in~$\Gamma$ that show that such a~$D$ cannot exist. 

Figures \ref{case1} and \ref{case2} display induced subgraphs of~$\Gamma$. 
The relevant cuts in those figures are the edges that cross the thick lines.
The cases we study are that~$D$ contains the dashed edges of the appropriate figure corresponding to the case, see Figures~\ref{case1} and~\ref{case2}.
For easier reference we call the two larger vertices the \emph{central vertices}.

{\bf Case 1:}
We now consider Figure \ref{case1}, so we assume that the edges from the central vertices into the `upper' side are one going to the left and the other to the right. 
First we note that the cut~$F$ ensures that the curvy  edge between the central vertices is not contained in~$D$. 
Also note that~$F$ ensures that the remaining two edges leaving the central vertices must go to the `lower' side of Figure \ref{case1}.
As the cuts~$B$ and~$C$ have to meet an even number of edges of~$D$ we may, due to symmetry, assume that the dotted edge is also contained in~$D$. 
This yields the contraction that the cut~$A$ now cannot meet any edge of~$D$.

\begin{figure}[ht]
\centering
\includegraphics[]{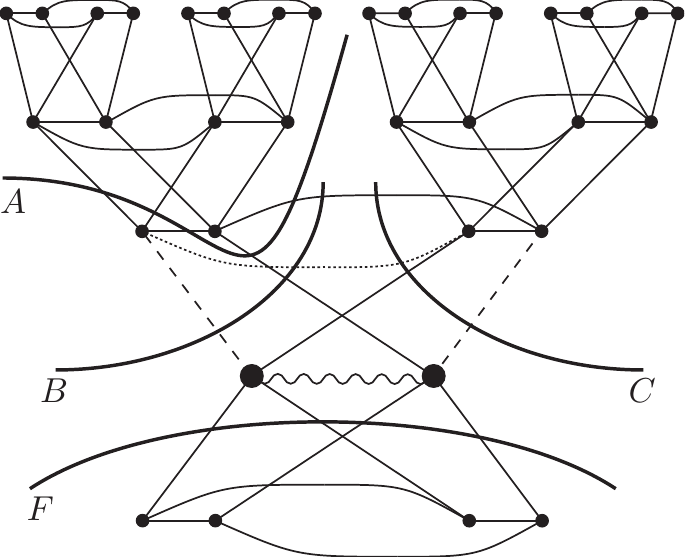}
\caption{Case 1}
\label{case1}
\end{figure}

{\bf Case 2:}
This case is very analog to Case 1. 
Again we may assume that the there are two edges leaving the central into the `upper' and the  `lower' side, each. 
The cut~$C$ ensures that~$D$ must contain both dotted edges. 
But this again yields the contraction that~$A$ cannot meet any edge in~$D$.

\begin{figure}[ht]
\centering
\includegraphics[]{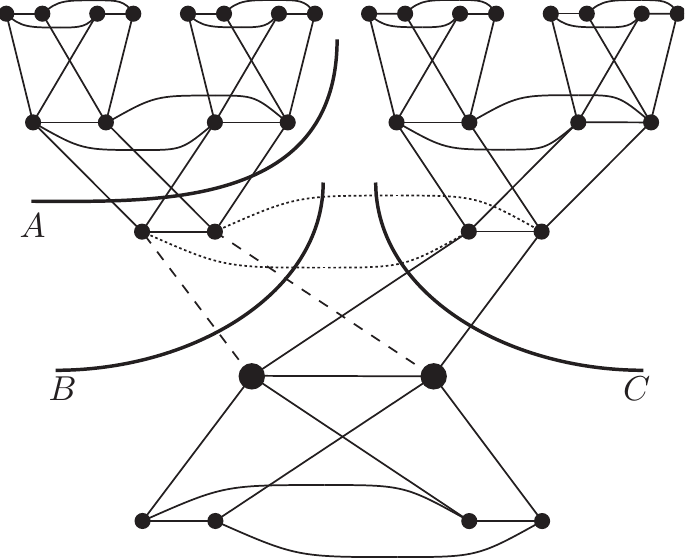}
\caption{Case 2}
\label{case2}
\end{figure}

It remains to show that~$G_1 \free_{\bbZ_2} G_2$ cannot be obtained as a free product with amalgamation over subgroups of size~$k$  of more than~$k$ groups.
If~$G_1 \free_{\bbZ_2} G_2$ were fulfilling the premise of Problem~\ref{conj1} then there would be a finite~${W \subset V(\Gamma)}$, say~${|W| =k}$, such that~${\Gamma \setminus W}$ has more than~$k$ components. 

We will now use induction on the size of~$W$. 
For a contraction we assume that such a set~$W$ exists.
For that we now introduce some notation to make the following arguments easier. 
In the following we will consider each group element as its corresponding vertex in~$\Gamma$. 
As~$\Gamma$ is transitive we may assume that~$1$ is contained in~$W$. 
Furthermore we may even assume that no vertex which starts with~$c$ is contained in~$W$.
Let~$X_i$ be the set of vertices in~$\Gamma$ that have distance exactly~$i$ from~$\{0,a\}$. 
We set~$W_i \defi X_i \cap W$. 
For~$x_i \in W_i$ let~$x_i^-$ be its neighbor in~$X_{i-1}$, note that this is unique. 
For a vertex~$x \in X_i$ let~$\bar{x}$ be the neighbor of~$x$ in~$X_i$ which is not~$xa$, note this will always be either~$xb$ or~$xc$. 
For a set~$Y$ of vertices of~$\Gamma$ let~$C_Y$ be the number of components of~${\Gamma \setminus Y}$. 

As~$\Gamma$ is obviously 2-connected the induction basis for~$|W| = 0$ or~$|W| =1$ holds trivially. 
\vspace{0,5cm}

We now assume that~$|W| = k$ and that for each~$W^\prime$ with~$|W^\prime | \leq | W |-1$ we know that~${C_{W^\prime} \leq |W^\prime|}$. 
In our argument we will remove sets of vertices of size~$l$ from~$W$ while decreasing~$C_W$ by at most~$l$.

Let~$x \in W$ be a vertex with the maximum distance to~$\{1,a\}$ in~$\Gamma$.
Say~$x \in X_j$ and define~${W_j \defi W \cap X_j}$. 
 
First we suppose that~$xa \notin W$. 
It is easy to see that removing~$x$ from~$W$ does not change~$C_W$. 
So we may assume that for every~${x \in W_j}$ the vertex~$xa$ is also in~$W$. 

We will now study all the possible cases for~$x^-$ and~$x^{-}a$. 

\begin{enumerate}
\item[Case 1:] $x^-, x^{-}a \notin W$: When we remove~$x$ and~$xa$ from~$W$ change~$C_W$ by at most~$1$ but reduced~$W$ by two elements, so we can assume this never happens.  
\item[Case 2:] $x^-, x^{-}a \in W$: When we remove~$x$ and~$xa$ from~$W$ we again decreased~$C_W$ by at most one while reducing~$|W|$ by two, so we again assume that this never happens.  
\item[Case 3:] $x^- \in W, x^{-}a \notin W$: By removing~$x,xa$ from~$W$ may reduce that~$C_W$ is reduced by one or two.
But as we have just decreased~$|W|$ by two both of these cases are fine and we are done.  
\end{enumerate}

\subsection{Closing Words}
\label{closing}
We still believe that it should be possible to find a condition on the size of the subgroup~$H$ to amalgamate over relative to the index of~$H$ in~$G_1$ and~$G_2$ such that the free product with amalgamation of~$G_1$ and~$G_2$ over~$H$ contains a Hamilton circle for the standard generating set. 
In addition it might be necessary to require some condition on the group~$G_1 / H$. 
We conjecture the following:

\begin{conj}
\label{conj:Multi}
There exists a function~$f: \bbN \rightarrow \bbN$ such that if~$G_1 = \langle S_1 \rangle$ and~$G_2 = \langle S_2 \rangle$ are finite groups with following properties: 
\begin{enumerate}[{\rm (i)}]
\item $[G_1 : H]=k$ and $[G_2 :H] =2$.
\item $|H| \geq f(k)$.
\item Each subgroup of~$H$ is normal in~$G_1$ and~$G_2$.
\item $\Gamma(G_1 /H, S/ H)$ contains a Hamilton cycle.  
\end{enumerate}
Then $\Gamma(G_1 \free_H G_2, S_1 \cup S_2)$ contains a Hamilton circle.  
\end{conj}

\bibliographystyle{plain}
\bibliography{collective.bib}

\begin{thebibliography}{10}

\bibitem{BabaiLovasz}
L.~Babai.
\newblock Automorphism groups, isomorphism, reconstruction.
\newblock {\em Handbook of Combinatorics}, Vol. 2:1447--1540, 1996.

\bibitem{bogo}
O.~Bogopolski.
\newblock {\em Introduction to group theory}.
\newblock European Mathematical Society (EMS). EMS Textbooks in Mathematics,
  Z\"urich, 2008.

\bibitem{ChenQuimpo}
C.~C. Chen and N.~Quimpo.
\newblock Hamilton cycles in cayley graphs over hamiltonian groups.
\newblock {\em Research Report}, No. 80, 1983.

\bibitem{diestelBook10noEE}
R.~Diestel.
\newblock {\em {Graph Theory}}.
\newblock Springer, 4th edition, 2010.

\bibitem{TopSurveyI}
R.~Diestel.
\newblock Locally finite graphs with ends: a topological approach. {I}.\
  {B}asic theory.
\newblock {\em Discrete Math.}, 311:1423--1447, 2011.

\bibitem{DiestelJungMoeller}
R.~Diestel, H.~A. Jung, and R.~G. M\"oller.
\newblock On vertex transitive graphs of infinite degree.
\newblock {\em Archiv der Mathematik}, 60(6):591, 1993.

\bibitem{Ends}
R.~Diestel and D.~K{\"u}hn.
\newblock Graph-theoretical versus topological ends of graphs.
\newblock {\em J.~Combin.\ Theory (Series B)}, 87:197--206, 2003.

\bibitem{Dragan}
M.~Dragan.
\newblock Hamiltonian circuits in cayley graphs.
\newblock {\em Discrete Math.}, 46(no. 1):49--54, 1983.

\bibitem{Durnberger1983}
E.~Durnberger.
\newblock Connected cayley graphs of semi-direct products of cyclic groups of
  prime order by abelian groups are hamiltonian.
\newblock {\em Discrete Math.}, 46:55--68, 1983.

\bibitem{AgelosFleisch}
A.~Georgakopoulos.
\newblock Infinite hamilton cycles in squares of locally finite graphs.
\newblock {\em Advances in Mathematics}, 220:670--705, 2009.

\bibitem{hamlehpott}
M.~Hamann, F.~Lehner, and J.~Pott.
\newblock Extending cycles locally to hamilton cycles.
\newblock {\em Electronic Journal of Combinatorics}, 23(Paper 1.49), 2016.

\bibitem{Karl1}
K.~Heuer.
\newblock A sufficient condition for hamiltonicity in locally finite graphs.
\newblock {\em Europ. J. Combinatorics}, pages 97--140, 2015.

\bibitem{Karl2}
K.~Heuer.
\newblock A sufficient local degree condition for hamiltonicity in locally
  finite claw-free graphs.
\newblock {\em European Journal of Combinatorics}, 55 Issue C:82--99, 2016.

\bibitem{commutatorsubgroup}
K.~Keating and D.~Witte.
\newblock On hamiltonian cycle in cayley graphs in groups with cyclic
  commutator subgroup.
\newblock In {\em Cycles in Graphs}, volume 115, pages 89--102. North-Holland,
  1985.

\bibitem{mei}
J.~Meier.
\newblock {\em Groups, Graphs and Trees: An Introduction to the Geometry of
  Infinite Groups}.
\newblock Cambridge, 2008.

\bibitem{nash}
C.St.J.A. Nash-Williams.
\newblock Abelian groups, graphs, and generalized knights.
\newblock {\em Proc. Camb. Phil. Soc.}, 55:232--238, 1959.

\bibitem{ScottWall}
P.~Scott and T.~Wall.
\newblock Topological methods in group theory, homological group theory.
\newblock {\em (Proc. Sympos., Durham, 1977), London Math. Soc. Lecture Note
  Ser., vol. 36}, Cambridge Univ. Press, Cambridge-New York:137--203, 1979.

\bibitem{scott}
W.~R. Scott.
\newblock {\em Group Theory}.
\newblock Prentice-Hall, Englewood Cliffs, NJ,, 1964.

\bibitem{Marusic}
D.~Maru\v si\v c.
\newblock Hamiltonian circuits in cayley graphs.
\newblock {\em Discrete Math}, 46:49--54, 1983.

\bibitem{wittedigraphs}
D.~Witte.
\newblock On hamiltonian circuits in cayley diagrams.
\newblock {\em Discrete Math.}, 38:99--108, 1982.

\bibitem{p-group}
D.~Witte.
\newblock Cayley digraphs of prime-power order are hamiltonian.
\newblock {\em J. Combin. Theory Ser. B}, 40:107--112, 1986.

\bibitem{wittesurvey}
D.~Witte and J.A. Gallian.
\newblock A survey: Hamiltonian cycles in cayley graphs.
\newblock {\em Discrete Math}, 51:293--304, 1984.

\end{thebibliography}

   \end{document}